\batchmode
\makeatletter
\def\input@path{{\string"/Users/russw/Documents/Research/mypapers/Frankl's conjecture for sg lattices/\string"/}}
\makeatother
\documentclass[12pt,oneside,english]{amsart}
\usepackage[T1]{fontenc}
\usepackage[latin9]{inputenc}
\usepackage{geometry}
\usepackage{babel}
\usepackage{url}
\usepackage{amsthm}
\usepackage{amssymb}
\usepackage[numbers]{natbib}
\usepackage[unicode=true,pdfusetitle,
 bookmarks=true,bookmarksnumbered=false,bookmarksopen=false,
 breaklinks=false,pdfborder={0 0 1},backref=false,colorlinks=false]
 {hyperref}
\usepackage{breakurl}

\makeatletter
\numberwithin{equation}{section}
\numberwithin{figure}{section}
\usepackage{enumitem}		
\theoremstyle{plain}
\newtheorem{thm}{\protect\theoremname}[section]
  \theoremstyle{plain}
  \newtheorem{conjecture}[thm]{\protect\conjecturename}
  \theoremstyle{plain}
  \newtheorem{cor}[thm]{\protect\corollaryname}
  \theoremstyle{plain}
  \newtheorem{prop}[thm]{\protect\propositionname}
  \theoremstyle{plain}
  \newtheorem{lem}[thm]{\protect\lemmaname}

\usepackage{ae}
\usepackage{aecompl}

\makeatletter
\newcommand{\setlabel}[1]{\def\@currentlabel{#1}}
\makeatother

\makeatother

  \providecommand{\conjecturename}{Conjecture}
  \providecommand{\corollaryname}{Corollary}
  \providecommand{\lemmaname}{Lemma}
  \providecommand{\propositionname}{Proposition}
\providecommand{\theoremname}{Theorem}

\begin{document}
\global\long\def\normalin{\mathrel{\triangleleft}}

\global\long\def\semidirect{\rtimes}

\global\long\def\comma{{,}}

\global\long\def\bot{\hat{0}}

\global\long\def\top{\hat{1}}

\global\long\def\subsetdot{\mathrel{\subset\!\!\!\!{\cdot}\,}}

\global\long\def\dotsupset{\mathrel{\supset\!\!\!\!\!\cdot\,\,}}

\global\long\def\height{\operatorname{height}}

\title{Frankl's Conjecture for subgroup lattices}

\author{Alireza Abdollahi, Russ Woodroofe, and Gjergji Zaimi}

\thanks{The first author was supported in part by grant No. 94050219 from
School of Mathematics, Institute for Research in Fundamental Sciences
(IPM). The first author was additionally financially supported by
the Center of Excellence for Mathematics at the University of Isfahan.}

\address{Department of Mathematics, University of Isfahan, Isfahan, 81746-73441,
Iran\vspace{-0.2cm}
}

\address{School of Mathematics, Institute for Research in Fundamental Sciences
(IPM), P.O. Box 19395-5746, Tehran, Iran}

\email{a.abdollahi@math.ui.ac.ir}

\urladdr{\url{http://sci.ui.ac.ir/~a.abdollahi}}

\address{Department of Mathematics \& Statistics, Mississippi State University,
MS 39762}

\email{rwoodroofe@math.msstate.edu}

\urladdr{\url{http://rwoodroofe.math.msstate.edu/}\bigskip{}
}

\email{gjergjiz@gmail.com}
\begin{abstract}
We show that the subgroup lattice of any finite group satisfies Frankl's
Union-Closed Conjecture. We show the same for all lattices with a
modular coatom, a family which includes all supersolvable and dually
semimodular lattices. A common technical result used to prove both
may be of some independent interest.
\end{abstract}

\maketitle

\section{\label{sec:Introduction}Introduction}

\subsection{Frankl's Conjecture}

All groups and lattices considered in this paper will be finite. We
will examine the following conjecture, attributed to Frankl from 1979.
\begin{conjecture}[Frankl's Union-Closed Conjecture]
\label{conj:Frankl}  If $L$ is a lattice with at least $2$ elements,
then there is a join-irreducible $a$ with $\left|[a,\top]\right|\leq\frac{1}{2}\left|L\right|$.
\end{conjecture}
There are a number of different equivalent forms of this conjecture.
The original form that Frankl considered involved a related condition
for families of sets that are closed under intersection. The first
appearance in print was in the conference proceedings \citep{Rival:1985},
arising from its mention by Duffus in a problem session. Three forms
of the problem are given in \citep{Rival:1985}: a statement about
families of sets closed under union, Frankl's original form, and the
lattice statement as we have here. Conjecture~\ref{conj:Frankl}
appears as a 5-difficulty problem in \citep{Stanley:1999}, where
it is called a ``diabolical'' problem. See \citep{Bruhn/Schaudt:2015}
for further information and history. The conjecture is currently the
subject of a Polymath project \citep{Polymath11}.

We will henceforth refer to Conjecture~\ref{conj:Frankl} as Frankl's
Conjecture. We will focus on the lattice form. If we wish to refer
to the join-irreducible $a$ satisfying the required condition, we
will say $L$ satisfies Frankl's Conjecture with $a$.

Frankl's Conjecture, while open in general, is known to hold for many
families of lattices. Poonen in \citep{Poonen:1992} proved and generalized
remarks of Duffus from \citep{Rival:1985}: that the conjecture holds
for distributive lattices, and for relatively complemented (including
geometric) lattices. Reinhold \citep{Reinhold:2000} showed the conjecture
to hold for dually semimodular lattices (see also \citep{Abe/Nakano:2000}).
Whether the conjecture holds for semimodular lattices is in general
unknown, but Czédli and Schmidt in \citep{Czedli/Schmidt:2008} verified
it for semimodular lattices that have a high ratio of elements to
join-irreducibles. 

We remark that Blinovsky has an arXiv preprint which claims to settle
the Frankl Conjecture. However: his argument is difficult to follow,
and has gone through a large number of arXiv versions in a short time.
Moreover, he has also claimed to solve several other difficult conjectures
in a short period, using the same technique. There does not seem at
this time to be a consensus that his proof is correct.

\subsection{Subgroup lattices}

Recall that for a group $G$, the \emph{subgroup lattice} of $G$
is the set $L(G)$ of all subgroups of $G$, ordered by inclusion. 

Our first main theorem verifies that Frankl's Conjecture holds for
subgroup lattices. 
\begin{thm}
\label{thm:MainThm}If $G$ is a group and $L(G)$ is the subgroup
lattice of $G$, then $L(G)$ satisfies Conjecture~\ref{conj:Frankl}.
\end{thm}
Subgroup lattices of groups form a large family of lattices. Indeed,
it is an important open question (first asked by Pálfy and Pudlák
\citep{Palfy/Pudlak:1980}) as to whether every finite lattice occurs
as an interval in the subgroup lattice of some finite group. Although
most experts on the topics appear to believe the answer to the Pálfy-Pudlák
question to be negative, progress has been somewhat limited. Indeed,
the problem is difficult \citep{Baddeley/Lucchini:1997} even for
lattices of height 2! See \citep{Aschbacher:2013} and its references
for further discussion of the Pálfy-Pudlák question and attempts to
disprove it.

In light of the question of Pálfy and Pudlák, it would be highly interesting
to settle Frankl's Conjecture in intervals of the form $[H,G]$ of
$L(G)$. We cannot do this in general, but give group-theoretic sufficient
conditions. We will state these conditions carefully in Corollary~\ref{cor:TechSgIntervals}.
We also verify that Frankl's Conjecture holds for every interval in
a solvable group in Corollary~\ref{cor:SolvIntervals}.

\subsection{Modular elements, subgroup lattices, and Frankl's Conjecture}

An essential tool in the proof of Theorem~\ref{thm:MainThm} also
has applications to many other lattices. For this reason, we give
it in a quite general form.

An element $m$ of a lattice $L$ is \emph{left-modular} if for every
$a<b$ in $L$, the expression $a\vee m\wedge b$ can be written without
parentheses. That is, if $a\vee(m\wedge b)=(a\vee m)\wedge b$ for
every $a<b$. We show:
\begin{thm}[Main Technical Theorem]
 \label{thm:TechLattices}Let $L$ be a lattice, let $m\in L\setminus\{\top\}$
be left-modular, and let $x,y\in L$ be (not necessarily distinct)
join-irreducibles. If $m\vee x\vee y=\top$, then $L$ satisfies Frankl's
Conjecture with either $x$ or $y$.
\end{thm}
It follows from the well-known Dedekind Identity (see Section~\ref{sub:ModSglat}
below) that any normal subgroup $N$ of $G$ is left-modular in $L(G)$.
It is straightforward to see that a subgroup $X$ is a join-irreducible
in $L(G)$ if and only if $X$ is cyclic of prime-power order. Thus,
we obtain the following as an easy consequence of Theorem~\ref{thm:TechLattices}.
\begin{cor}
\label{cor:TechGroups}If $G$ is a group with $N\normalin G$, and
$G/N$ is generated by at most two elements of prime-power order,
then $L(G)$ satisfies Frankl's Conjecture.
\end{cor}
Theorem~\ref{thm:MainThm} will follow by combining Corollary~\ref{cor:TechGroups}
with results on finite simple groups. 

We similarly obtain a relative version for upper intervals in groups.
The statement is somewhat harder to work with, as we are not aware
of any short description for join-irreducibles in intervals of subgroup
lattices.
\begin{cor}
\label{cor:TechSgIntervals}Let $G$ be a group and $H$ be a subgroup.
If $X$ and $Y$ are join-irreducibles of the interval $[H,G]$, and
$N\normalin G$ is such that $HN<G$ but $HN\vee X\vee Y=G$, then
the interval $[H,G]$ satisfies Frankl's Conjecture.
\end{cor}

\subsection{The Averaged Frankl's Condition}

A related question to Frankl's Conjecture asks for which lattices
the average size over a join-irreducible element (other than $\hat{0}$)
is at most $\frac{1}{2}\left|L\right|$. We call this condition the
\emph{Averaged Frankl's Condition}. The Averaged Frankl's Condition
does not hold for all lattices, but is known to hold for lattices
with a large ratio of elements to join-irreducibles \citep{Czedli:2009}.
The condition obviously holds for uncomplicated subgroup lattices
such as $L(\mathbb{Z}_{p^{n}})$ or $L(\mathbb{Z}_{p}^{n})$. Indeed,
our techniques allow us to show a stronger condition for a restrictive
class of groups.
\begin{prop}
\label{prop:ComplementedGroups}If $G$ is a supersolvable group so
that all Sylow subgroups of $G$ are elementary abelian, then $G$
satisfies Frankl's Conjecture with any join-irreducible $X$.
\end{prop}
\noindent Supersolvable groups with elementary abelian subgroups are
also known as \emph{complemented groups}, and were first studied by
Hall \citep{Hall:1937}. We don't know whether the subgroup lattices
of arbitrary groups always satisfy the Averaged Frankl's Condition.

\subsection{Other lattices}

Left-modular elements also occur in lattices from elsewhere in combinatorics.
A situation that is both easy and useful is:
\begin{cor}
If a lattice $L$ has a left-modular coatom $m$, then $L$ satisfies
Frankl's Conjecture.\end{cor}
\begin{proof}
If $\top$ is a join-irreducible, then the result is trivial. Otherwise,
there is some join-irreducible $x$ such that $m\vee x=\top$, and
we apply Theorem~\ref{thm:TechLattices}.
\end{proof}
There has been much study of classes of lattices that have a left-modular
coatom. Dually semimodular lattices have every coatom left-modular,
so we recovery the earlier-mentioned result \citep{Reinhold:2000}
that such lattices satisfy Frankl's Conjecture. We also obtain the
new result that \emph{supersolvable} and \emph{left-modular} lattices
(those with a maximal chain consisting of left-modular elements) satisfy
Frankl's Conjecture. See e.g. \citep{McNamara/Thomas:2006} for background
on supersolvable lattices.

Still more generally, the \emph{comodernistic} lattices recently examined
by the second author and Schweig \citep{Schweig/Woodroofe:2016UNP}
are those lattices with a left-modular coatom on every interval. This
class of lattices includes all supersolvable, left-modular, and dually
semimodular lattices. It also includes other large classes of examples,
including subgroup lattices of solvable groups, and $k$-equal partition
lattices.
\begin{thm}
\label{thm:ComodLats}Comodernistic lattices (including supersolvable,
left-modular, and dually semimodular lattices) satisfy Frankl's Conjecture.
\end{thm}
Subgroup lattices of solvable groups are one family of examples of
comodernistic lattices \citep[Theorem 1.7]{Schweig/Woodroofe:2016UNP}.
That is, every interval in the subgroup lattice of a solvable group
has a left-modular coatom. It follows immediately that:
\begin{cor}
\label{cor:SolvIntervals}If $G$ is a solvable group, then every
interval in $L(G)$ satisfies Frankl's Conjecture.
\end{cor}
Since the $\bot$ element of any lattice is left-modular, Theorem~\ref{thm:TechLattices}
also yields the following:
\begin{cor}
If $L$ is a lattice such that $\top=x\vee y$ for join-irreducibles
$x,y$, then $L$ satisfies Frankl's Conjecture.
\end{cor}

\subsection{Organization}

In Section~\ref{sec:Groups} we will discuss the group-theoretic
aspects of the problem. We will complete the proof of Corollary~\ref{cor:TechGroups}
and Theorem~\ref{thm:MainThm}, pending only on the proof of Theorem~\ref{thm:TechLattices}.
In Section~\ref{sec:ProofTechThm}, we will prove Theorem~\ref{thm:TechLattices}
and generalizations, as well as Proposition~\ref{prop:ComplementedGroups}.

\section*{Acknowledgements}

We would like to thank the administrators and community of MathOverflow,
which brought us together to work on the problem \citep{MO-Question}.
We also thank Tobias Fritz and Marco Pellegrini for carefully reading
earlier drafts. Marco Pellegrini in particular provided us with additional
background on generation of Suzuki groups, including the useful reference
to \citep{Evans:1987}.

\section{\label{sec:Groups}Groups, generation, and subgroup lattices}

The main purpose of this section is to prove Theorem~\ref{thm:MainThm},
as we do in Section~\ref{sub:ProofMainThm}. We first begin with
some basic background on the combinatorics of subgroup lattices.

\subsection{\label{sub:ModSglat}Modular elements in subgroup lattices, and the
details of Corollary~\ref{cor:TechGroups}}

The Dedekind Identity is often assigned as an exercise \citep[Exercise 2.9]{Isaacs:1994}
in a graduate algebra course:
\begin{lem}[Dedekind Modular Identity]
 If $H,K,N$ are subgroups of a group $G$ such that $H\leq K$,
then $H(N\cap K)=HN\cap K$.
\end{lem}
It is also well known that $HN$ is a subgroup of $G$ if and only
if $HN=NH=H\vee N$. These conditions are obviously satisfied when
$N$ is a normal subgroup, and are sometimes otherwise satisfied. 

It is thus immediate from the Dedekind Identity that whenever $HN$
is a subgroup, we also have that $N$ satisfies the modular relation
with $H$ and any $K>H$. In particular, we recover our earlier claim
that normal subgroups are left-modular in $L(G)$.

The proof of Corollary~\ref{cor:TechGroups} follows from this fact,
together with another routine exercise: If $\overline{x}$ and $\overline{y}$
are elements of prime-power order in $G/N$, then there are $x,y\in G$
of prime-power order such that $\overline{x}=Nx$, $\overline{y}=Ny$
\citep[Exercise 3.12]{Isaacs:1994}. In particular, the modular subgroup
$N$ and the join-irreducibles $\left\langle x\right\rangle $ and
$\left\langle y\right\rangle $ satisfy the conditions of Theorem~\ref{thm:TechLattices}.

Corollary~\ref{cor:TechSgIntervals} follows by a similar argument.

\subsection{\label{sub:ProofMainThm}Proof of Theorem~\ref{thm:MainThm}}

We prove Theorem~\ref{thm:MainThm} by combining Corollary~\ref{cor:TechGroups}
with facts about finite simple groups. King recently proved in \citep{King:2016UNP}:
\begin{thm}[Prime Generation Theorem \citep{King:2016UNP}]
\label{thm:PrimeGen}  If $G$ is any nonabelian finite simple group,
then $G$ is generated by an involution and an element of prime order.
\end{thm}
Whenever $N$ is a maximal normal subgroup of $G$, the quotient $G/N$
is simple. Of course, abelian simple groups are generated by a single
element of prime order. Nonabelian simple groups are handled by Theorem~\ref{thm:PrimeGen}.
Theorem~\ref{thm:MainThm} now follows from Corollary~\ref{cor:TechGroups}.

\subsection{Overview of generation of simple groups by elements of prime order}

The substantive work of King \citep{King:2016UNP} in proving Theorem~\ref{thm:PrimeGen}
builds on a large body of preceding work. We will briefly survey some
history and mathematical details. We assume basic knowledge of the
Classification of Finite Simple Groups in this discussion, but will
not assume any such elsewhere in the paper.

A group $G$ is said to be \emph{$(p,q)$-generated} if $G$ is generated
by an element of order $p$ and one of order $q$. The case of $(2,3)$-generation
is particularly well-studied in the literature, as such groups are
the quotients of the infinite group $PSL_{2}(\mathbb{Z})$. In addition
to the references below, see e.g. \citep{Pellegrini/Tamburini:2015,Vsemirnov:2011}.

The following has been known to hold for some time.
\begin{prop}
\label{prop:PrimeGen235} With at most finitely many exceptions, every
nonabelian finite simple group is either $(2,3)$- or $(2,5)$-generated.
\end{prop}
We summarize the history behind Proposition~\ref{prop:PrimeGen235}.
The alternating group $A_{n}$ was shown to be $(2,3)$-generated
by Miller \citep{Miller:1901} for $n\neq6,7,8$; while $A_{6},A_{7}$
and $A_{8}$ are easily seen to be $(2,5)$-generated. Excluding the
groups $PSp_{4}(q)$, all but finitely many of the classical groups
are $(2,3)$-generated by work of Liebeck and Shalev \citep{Liebeck/Shalev:1996}.
In the same paper \citep{Liebeck/Shalev:1996}, the authors showed
that, excluding finitely many exceptions, in characteristic $2$ or
$3$ the groups $PSp_{4}$ are $(2,5)$-generated. Cazzola and Di
Martino in \citep{Cazzola/DiMartino:1993} showed $PSp_{4}$ to be
$(2,3)$-generated in all other characteristics. Lübeck and Malle
\citep{Lubeck/Malle:1999} (building on earlier work by Malle \citep{Malle:1990,Malle:1995})
showed all simple exceptional groups excluding the Suzuki groups to
be $(2,3)$-generated. Evans \citep{Evans:1987} showed the Suzuki
groups to be $(2,p)$-generated for any odd prime $p$ dividing the
group order, and in particular to be $(2,5)$-generated. Proposition~\ref{prop:PrimeGen235}
now follows by combining the results enumerated here with the Classification
of Finite Simple Groups.

We caution that $PSU_{3}(3^{2})$ is known not to be $(2,3)$-generated
\citep{Wagner:1978}, and since it has order $\left|PSU_{3}(3^{2})\right|=2^{5}\cdot3^{3}\cdot7$,
the group is certainly not $(2,5)$-generated either. 

King's proof of Theorem~\ref{thm:PrimeGen} proceeds by showing that
every classical simple group $G$ is either $(2,3)$-, $(2,5)$-,
or $(2,r)$-generated, where $r$ is a so-called Zsigmondy prime for
$G$.

\section{\label{sec:ProofTechThm}Proof of Theorem~\ref{thm:TechLattices}}

Since $m\neq\top$, we see that $x\vee y\not\leq m$. If $x\leq m$,
then we may replace the triple $m,x,y$ with $m,y,y$ while still
meeting the conditions of the theorem. Thus, we may suppose without
loss of generality that neither $x$ nor $y$ is on the interval $[\bot,m]$.

Suppose without loss of generality that $[x,\top]$ has at most as
many elements as $[y,\top]$. We will show that $\left|[x,\top]\right|\leq\frac{1}{2}\left|L\right|$
by constructing an injection from $[x,\top]$ to its complement in
$L$. 

For the first part of the injection, since $\left|[x,\top]\right|\leq\left|[y,\top]\right|$,
there is some injection that maps 
\[
\varphi_{1}:[x,\top]\setminus[x\vee y,\top]\rightarrow[y,\top]\setminus[x\vee y,\top].
\]

For the second part of the injection, we look at the interval $[x\vee y,\top]$.
(We notice that if $x=y$, then $x\vee y=x=y$, and this will cause
no trouble in what follows.) We map 
\begin{align*}
\varphi_{2}:[x\vee y,\top] & \rightarrow[\bot,m]\\
\alpha & \mapsto m\wedge\alpha.
\end{align*}
As $x\vee y\vee(m\wedge\alpha)=(x\vee y\vee m)\wedge\alpha=\top\wedge\alpha=\alpha$
by left-modularity, the map $\varphi_{2}$ is an injection. Since
$x\not\leq m$, the image of $\varphi_{2}$ is contained in the complement
of $[x,\top]$.

The two maps $\varphi_{1},\varphi_{2}$ have disjoint domains. Combining
them yields the desired injection.

\subsection{Generalizations}

Examining our proof of Theorem~\ref{thm:TechLattices}, we observe
that we do not use the full power of left-modularity, but only that
$m$ satisfies the left-modular relation for any $\alpha>x\vee y$.
Thus, we have actually proved the following generalization:
\begin{prop}
\label{prop:TechGeneralized}Let $L$ be a lattice, and let $x,y\in L$
be (not necessarily distinct) join-irreducibles. If $m\in L\setminus\{\top\}$
satisfies $(x\vee y\vee m)\wedge\alpha=(x\vee y)\vee(m\wedge\alpha)$
for any $\alpha>x\vee y$, and $m\vee x\vee y=\top$, then $L$ satisfies
Frankl's Conjecture with either $x$ or $y$.
\end{prop}
While the statement of Proposition~\ref{prop:TechGeneralized} appears
notably more complicated than that of Theorem~\ref{thm:TechLattices},
it yields a reasonably uncomplicated corollary for intervals in subgroup
lattices.
\begin{cor}
\label{cor:SgintGeneralization}Let $G$ be a group, let $H<G$, and
let $X,Y$ be join-irreducibles of $[H,G]$. If there is a subgroup
$K$ with $H<K<G$ such that $K(X\vee Y)=G$, then the interval $[H,G]$
satisfies Frankl's Conjecture.
\end{cor}
We in particular are now able to prove Proposition~\ref{prop:ComplementedGroups}.
\begin{proof}[Proof (of Proposition~\ref{prop:ComplementedGroups}).]
 It follows by a theorem of Hall \citep{Hall:1937} that for every
subgroup $H$ in $G$, there is some subgroup $K$ such that $KH=G$
and $H\cap K=1$. The result follows by combining the theorem of Hall
with Corollary~\ref{cor:SgintGeneralization}.
\end{proof}
\bibliographystyle{hamsplain}
\bibliography{0_Users_russw_Documents_Research_mypapers_Frankl_s_conjecture_for_sg_lattices_Master,1_Users_russw_Documents_Research_mypapers_Frankl_s_conjecture_for_sg_lattices_Webpages}

\end{document}